\documentclass[preprint,12pt]{elsarticle}

\usepackage{amssymb,amsmath}

\newtheorem{remark}{Remark}
\newtheorem{proposition}{Proposition}
\newtheorem{assumption}{Assumption}
\newtheorem{theorem}{Theorem}
\newtheorem{proof}{Proof}

\journal{Applied Mathematics and Computation}

\begin{document}

\begin{frontmatter}



\title{Goodness-of-fit test for the bivariate Hermite distribution}


\author{Pablo Gonz\'alez-Albornoz}
\author{Francisco Novoa-Mu\~noz\corref{cor1}}

\address{Departamento de Estad\'istica, Universidad del B\'io-B\'io, Concepci\'on, Chile}
\cortext[cor1]{Coresponding author.} \ead{fnovoa@ubiobio.cl}

\begin{abstract}
This paper studies the goodness of fit test for the bivariate
Hermite distribution. Specifically, we propose and study a
Cram\'er-von Mises-type test based on the empirical probability
generation function. The bootstrap can be used to consistently
estimate the null distribution of the test statistics. A simulation
study investigates the goodness of the bootstrap approach for finite
sample sizes.
\end{abstract}




\end{frontmatter}


\section{Introduction}
\label{intro}
The counting data can appear in different circumstances. In the univariate configuration, the Hermite distribution (HD) is a linear combination of the form $Y=X_1+2X_2$, where $X_1$ and $X_2$ are independent Poisson random variables. The properties that distinguish the HD is to be flexible when it comes to modeling counting data that present a multimodality, along with presenting several zeros, which is called zero-inflation. It also allows modeling data in which the overdispersion is moderate; that is, the variance is greater than the expected value. It was McKendrick in \citep{McK} who modeled a phagocytic experiment (bacteria count in leukocytes) through the HD, obtaining a more satisfactory model than with the Poisson distribution. However, in practice, the bivariate count data arise in several different disciplines and  bivariate Hermite distribution (BHD) plays an important role, having superinflated data. For example, the accident number on two different periods \citep{Cress}.\\

Testing the goodness of fit (gof) of observations given with a probabilistic model is a crucial aspect of data analysis. For the univariate case, we have only found a single test of gof, but for data that come from a generalized Hermite distribution (for a review, see Meintanis and Bassiakos in \citep{Meintanis}), but not from a HD. On the other hand, we did not find literature on gof tests for BHD.\\

The purpose of this paper is to propose and study a goodness-of-fit test for the bivariate Hermite Distribution that is consistent.\\

According to Novoa-Mu\~noz in \citep{NM19}, the probability
generating function (pgf) characterizes the distribution of a random
vector and can be estimated consistently by the empirical
probability generating function (epgf), the proposed test is a
function of the epgf. This statistical test compares the epgf of the
data with an estimator of the pgf of the BHD. As it is well known,
to establish the rejection region, we need to know the distribution
of the statistic test.

As for finite sample sizes the resulting test statistic is of the
Cram\'er-Von Mises type, it was not possible to calculate explicitly
the distribution of the statistic under null hypothesis. That is why
one uses simulation techniques. Therefore, we decided to use a null
approximation of the statistic by using a parametric bootstrap.

Because the properties of the proposed test are asymptotic (see, for example, \citep{Nov}) and with the purpose of evaluating the behavior of the test for sample of finite size, a simulation study was carried out.\\

The present work is ordered as follows. In section \ref{sec:2} we present some preliminary results that will serve us in the following chapters, the definition of the BHD with some of its properties is also given. In section \ref{sec:3}, the proposed statistic is presented. Section \ref{bootstrap} is devoted to showing the bootstrap estimator and its approximation to the null distribution of the statistic. Section \ref{NumericalResults} is dedicated to presenting the results of a simulation study,  power of a hypothesis test and the application to a set of real data.\\

Before ending this section, we introduce some notation: $\mathcal{F}_A\underset{\delta}{\wedge}\mathcal{F}_B$ denotes a mixture (compounding) distribution where $\mathcal{F}_A$ represents the original distribution and $\mathcal{F}_B$ the mixing distribution (i.e., the distribution of $\delta$) \citep{JKK05}; all vectors are row vectors and $x^\top$ is the transposed of the row vector $x$; %
for any vector $x,\, x_k$ denotes its $k$th coordinate, and $\|x\|$
its Euclidean norm; $\mathbb{N}_0=\{0,1,2,3,...\}$; $I\{A\}$ denotes
the indicator function of the set $A$; $P_{\theta}$ denotes the
probability law of the BHD with parameter $\theta$; $E_{\theta}$
denotes expectation with respect to the probability function
$P_{\theta}$; $P_{*}$ and $E_{*}$ denote the conditional probability
law and expectation, given the data $(X_1, Y_1), . . . , (X_n,
Y_n)$, respectively; all limits in this work are taken as $n
\rightarrow \infty; \mathop{\longrightarrow} \limits^{L}\ $ denotes
convergence in distribution;
$\mathop{\longrightarrow} \limits^{a.s.}\ $ denotes  almost sure
convergence; let $\{C_n\}$ be a sequence of random variables or
random elements and let $ \epsilon\in\mathbb{R}$, then
$C_n=O_{_P}(n^{-\epsilon})$ means that $n^{\epsilon}C_n$ is bounded
in probability, $C_n=o_{_P}(n^{-\epsilon})$ means that
$n^{\epsilon}C_n\ \mathop{\longrightarrow} \limits^{P} 0$ and
$C_n=o(n^{-\epsilon})$ means that $n^{\epsilon}C_n\
\mathop{\longrightarrow} \limits^{a.s.} 0$ and
$\mathcal{H}=L^2\left([0,1]^2,\varrho\right)$ denotes the separable
Hilbert space of the measurable functions $\varphi,
\varrho:[0,1]^2\rightarrow \mathbb{R}$ such that
$||\varphi||_{\mathcal{H}}^{2}=\int_0^1\int_0^1
\varphi^2({t})\,\varrho({t})d{t}<\infty$.

\section{Preliminaries}
\label{sec:2}
Several definitions for the BHD have been given (see, for example, Kocherlakota and Kocherlakota in \citep{Kocher}). In this paper we will work with the following one, which has received more attention in the statistical literature (see, for example, Papageorgiou et al. in \citep{Papa}; Kemp et al. in \citep{kkemp}).\\

Let $\boldsymbol{X}=(X_1,X_2)$ has the bivariate Poisson
distribution with the parameters $\delta\lambda_1$,
$\delta\lambda_2$ and $\delta\lambda_3$ (for more details of this
distribution, see for example, Johnson et al. in \citep{JKB97}),
then  $\boldsymbol{X} \underset{\delta}{\wedge} N(\mu,\sigma^2)$ has
the BHD. Kocherlakota in \citep{kocher} got its pgf which is given
by
\begin{equation}
v(t;\theta)
=exp\left(\mu \lambda+\frac{1}{2}\sigma^2
\lambda^2\right),\label{fgp}
\end{equation}
where $t=(t_1,t_2)$, $\theta=(\mu,\sigma^2,\lambda_1,\lambda_2,\lambda_3)$, $\lambda=\lambda_1(t_1-1)+\lambda_2(t_2-1)+\lambda_3(t_1t_2-1)$ and $\mu>\sigma^2(\lambda_i+\lambda_3)$, $i=1,2$.\\

From the pgf of the BHD, Kocherlakota and Kocherlakota
\citep{Kocher} obtained the probability mass function of the BHD,
which is given by
\[
f(r,s)=\frac{\lambda^r_1\lambda_2^s}{r!s!}M(\gamma)\sum_{k=0}^{min(r,s)}\binom{r}{k}\binom{s}{k}k!\,\xi^k
P_{r+s-k}(\gamma),
\]
where $M(x)$ is the moment-generating function of the normal
distribution, $P_r(x)$ is a polynomial of degree $r$ in $x$,
$\gamma=-(\lambda_1+\lambda_2+\lambda_3)$ and
$\xi=\frac{\lambda_3}{\lambda_1\lambda_2}$.

\begin{remark}
    If $\lambda_3=0$ then the probability function is reduced to
    \[
    f(r,s)=\frac{\lambda^r_1\lambda_2^s}{r!s!}\,M(-\lambda_1-\lambda_2)\, P_{r+s}(-\lambda_1-\lambda_2).
    \]
\end{remark}

\begin{remark}
    If $\boldsymbol{X}$ is a random vector that is bivariate Hermite distributed with parameter $\theta$ will be denoted $\boldsymbol{X}\sim BH(\theta)$, where $\theta\in \Theta$, and the parameter space is
    \[
    \Theta=\left\{(\mu,\sigma^2,\lambda_1,\lambda_2,\lambda_3)\in \mathbb{R}^5 /\mu>\sigma^2(\lambda_i+\lambda_3), \lambda_i> \lambda_3\geq 0,\, i=1,2\right\}.
    \]
\end{remark}

Let $\boldsymbol{X}_1=(X_{11},X_{12}),
\boldsymbol{X}_2=(X_{21},X_{22}), \ldots,
\boldsymbol{X}_n=(X_{n1},X_{n2})$ be independent and identically
distributed (iid) random vectors defined on a probability space
$(\Omega, \mathcal{A},\mathnormal{P})$ and taking values in
$\mathbb{N}_0^2$. In what follows, let
\[
v_n(t)=\displaystyle\frac{1}{n}\sum_{i=1}^n t_1^{X_{i1}}t_2^{X_{i2}}
\]
denote the epgf of $\boldsymbol{X}_1, \boldsymbol{X}_2, \ldots,\boldsymbol{X}_n$ for some appropriate $ W\subseteq  \mathbb{R}^2$.\\

In the next section we will develop our statistician and for this
the result given below will be fundamental, whose proof was
presented in \citep{Nov}.

\begin{proposition}\label{convfgp}
    Let $\textbf{X}_1,\ldots, \textbf{X}_n$  be iid from a random vector $\boldsymbol{X}=(X_1,X_2)\in \mathbb{N}_0^2$. Let $v(t)=E\!\left(t_1^{X_1} t_2^{X_2}\right)$  be the pgf of $\boldsymbol{X}$, defined on $W \subseteq \mathbb{R}^2$. Let $0\leq b_j\leq c_j<\infty,\ j=1, 2$, such that $Q=[b_1,c_1]\times [b_2,c_2] \subseteq W$, then
    \[
    \sup_{t\in Q} |v_n(t)-v(t)|\ \mathop{\longrightarrow}\limits^{a.s.} \ 0.
    \]
\end{proposition}

\section{The test statistic and its asymptotic null distribution}
\label{sec:3} Let $\textbf{X}_1=(X_ {11}, X_ {12}), \textbf{X}_2=(X_
{21}, X_ {22}), \ldots,\textbf{X}_n=(X_ {n1}, X_ {n2})$ be iid from
a random vector  $\boldsymbol{X}=(X_1,X_2)\in \mathbb{N}_0^2$. Based
on the sample
$\boldsymbol{X}_1,\boldsymbol{X}_2,\ldots,\boldsymbol{X}_n$, the
objective is to test the hypothesis
\[H_0 : (X_1,X_2)\sim BH(\theta), \ \text{for some}\ \theta\in\Theta,\]
against the alternative
$$H_1 : (X_1,X_2)\  \nsim \ BH(\theta), \ \forall\theta \in\Theta.$$

With this purpose, we will recourse to some of the properties of the pgf that allow us to propose the following statistical test.\\

According to Proposition \ref{convfgp}, a consistent estimator of
the pgf is the epgf. If $H_0$ is true  and $\hat{{\theta}}_n$ is a
consistent estimator of  ${\theta}$, then $v(t;\hat{{\theta}}_n)$
consistently estimates the population pgf. Since the distribution of
$\boldsymbol{X}=(X_1,X_2)$ is uniquely determined by its pgf,
$v(t)$, $t=(t_1, t_2)\in [0,1]^2$, a reasonable test for testing
$H_0$  should reject the null hypothesis for large values of
$V_{n,w}(\hat{{\theta}}_n)$ defined by
\begin{equation}\label{Estad-Rnw} V_{n,w}(\hat{{\theta}}_n)=\int_0^1\int_0^1 V^2_n(t;\hat{{\theta}}_n)w(t)dt,
\end{equation}
where  $$V_n(t;\theta)= \sqrt{n}\left\{v_n(t)-v(t;
\theta)\right\},$$
$\hat{\theta}_n=\hat{{\theta}}_n(\boldsymbol{X}_1,
\boldsymbol{X}_2,\ldots,\boldsymbol{X}_n)$  is a consistent
estimator of $\theta$ and $w(t)$ is a measurable weight function,
such that  $w(t)\geq 0, \ \forall t\in [0,1]^2$, and
\begin{equation} \label{int-funcion-peso}
\int_0^1\int_0^1 w(t)dt < \infty.
\end{equation}
The assumption (\ref{int-funcion-peso}) on $w$ ensures that the
double integral in (\ref{Estad-Rnw}) is finite for each fixed $n$.
Now, to determine what are large values of
$V_{n,w}(\hat{\theta}_n)$, we must calculate its null distribution,
or at least an approximation to it. Since the null distribution of
$V_{n,w}(\hat{\theta}_n)$ is unknown, we first try to estimate it by
means of its asymptotic null distribution. In order to derive it we
will assume that the estimator $\hat{\theta}_n$ satisfies the
following regularity condition.
\begin{assumption}\label{Sup1}
    Under $H_0$, if $\,\theta= (\mu,\sigma^2,\lambda_1,\lambda_2,\lambda_3)\in\Theta\,$ denotes the true parameter value, then
    \[\sqrt{n}\left(\hat{\theta}_n-\theta\right) =\frac{1}{\sqrt{n}}\,\sum_{i=1}^n \boldsymbol{\ell}\left(\boldsymbol{X}_{i}; \theta\right)+\mathbf{o}_{_P}(1),\]
    where $\boldsymbol{\ell}:\mathbb{N}^2_0\times \Theta \longrightarrow \mathbb{R}^5$ is such that $E_{\theta}\left\{\boldsymbol{\ell}\, (\boldsymbol{X}_{{1}}; \theta)\right\}=\mathbf{0}$ and $J(\theta)=E_{\theta}\left\{\boldsymbol{\ell} \left(\boldsymbol{X}_{{1}};\theta\right)^{\top}\boldsymbol{\ell} \left(\boldsymbol{X}_{{1}}; \theta\right)\right\}< \infty$.
\end{assumption}
Assumption 1 is fulfilled by most commonly used estimators, see \cite{Kocher} and \cite{PL88}.\\

The next result gives  the asymptotic null distribution of
$V_{n,w}(\hat{\theta}_n)$.

\begin{theorem}\label{ConvDebil}
    Let  $\boldsymbol{X}_{1},\ldots, \boldsymbol{X}_{n}$ be iid from $\boldsymbol{X}=(X_{1},X_{2})\sim BH(\theta)$. Suppose that
    Assumption 1  holds.
    Then
    \[
    V_{n,w}(\hat{\theta}_n)= ||W_n||_{_\mathcal{H}}^2+o_{_P}(1),
    \]
    where \
    $W_n(t)=\frac{1}{\sqrt{n}}\sum_{i=1}^n V^0(\boldsymbol{X}_i,{\theta}; {t})$, with
    $$V^0(\boldsymbol{X}_i,\theta;t)=t_1^{X_{i1}} t_2^{X_{i2}}-v(t;\theta)\left\{1+\left(\lambda,\frac{1}{2}\lambda^2,\eta(t_1-1),\eta(t_2-1),\eta (t_1t_2-1)\right)\boldsymbol{\ell} \left(\boldsymbol{X}_i;\theta\right)^{\top}\right\},$$
    $i=1, \ldots, n,\ \eta=\mu+\sigma^2\lambda$.\
    Moreover,
    \begin{equation}\label{lambdas}
    V_{n,w}(\hat{\theta}_n)\mathop{\longrightarrow}\limits^{L} \sum_{j\geq 1}\lambda_j\chi^2_{1j},
    \end{equation}
    where $\chi^2_{11},\chi^2_{12},\ldots$ are independent $\chi^2$ variates with one degree of freedom and the set $\{\lambda_j\}$ are the non-null eigenvalues of the operator $C(\theta)$ defined on the function space $\{\tau:\mathbb{N}_0^2\to \mathbb{R},  \text{such that} \ E_{\theta}\left\{\tau^2(\boldsymbol{X})\right\}<\infty,\forall \theta\in\Theta\}$, as follows
    \[
    C(\theta) \tau(\boldsymbol{x})= E_{\theta}\{h(\boldsymbol{x},\boldsymbol{Y};\theta) \tau(\boldsymbol{Y})\},
    \]
    where
    \begin{equation} \label{h}
    h(\boldsymbol{x},\boldsymbol{y};\theta)=\int_0^1\int_0^1  V^0(\boldsymbol{x}; \theta; t) V^0(\boldsymbol{y}; \theta; t)w(t)dt.
    \end{equation}
\end{theorem}
%
%
\begin{proof}
    By definition,
    $V_{n,w}(\hat{\theta}_n)= \|V_n(\hat{\theta}_n)\|_{_\mathcal{H}}^{2}$. Note that
    \begin{equation}\label{def_V}
    V_n(t;\hat{\theta}_n)=\frac{1}{\sqrt{n}}\sum_{i=1}^n V(\boldsymbol{X}_i;\hat{\theta}_n;t),\ \ \text{with}\ \ V(\boldsymbol{X}_i;\theta;t)=t_1^{X_{i1}} t_2^{X_{i2}}-v(t;\theta).
    \end{equation}
    By Taylor expansion of $V(\boldsymbol{X}_i;\hat{\theta}_n;t)$ around $\hat{\theta}_n=\theta$,
    \begin{equation}\label{Z-kn-aprox}
    V_n(t;\hat{\theta}_n)=\frac{1}{\sqrt{n}}\sum_{i=1}^n V(\boldsymbol{X}_i;\theta; t)+\frac{1}{n}\sum_{i=1}^nQ^{(1)}(\boldsymbol{X}_i;\theta; t)\,\sqrt{n}(\hat{\theta}_n-\theta)^{\top}+q_{n},
    \end{equation}
    where
    $q_{n}=\frac{1}{2\sqrt{n}}(\hat{\theta}_n-\theta) \sum_{i=1}^n Q^{(2)}(\boldsymbol{X}_i;\widetilde{\theta}; t) \,(\hat{\theta}_n-\theta)^{\top}$,
    $\widetilde{\theta}=\alpha \hat{\theta}_n+(1-\alpha)\theta$, for some $0<\alpha<1\,$, $Q^{(1)}(\boldsymbol{x};\vartheta; t)$ is the vector of the first derivatives and $Q^{(2)}(\boldsymbol{x};\vartheta; t)$ is the matrix of the second derivatives of $V(\boldsymbol{x};\vartheta; t)$ with respect to $\vartheta$.\\

    Thus, considering (\ref{int-funcion-peso}) results
    \begin{equation}\label{E-Qkj-acotada}
    E_{\theta}\left\{\left\|Q^{(1)}_{j}\left(\boldsymbol{X}_{1};\theta; t\right)\right\|_{_\mathcal{H}}^2\right\}<\infty,\ \ j=1,2,\ldots,5.
    \end{equation}
    Using the Markov inequality and (\ref{E-Qkj-acotada}), we have
    \begin{align}
    &P_{\theta}\left[\left\| \frac{1}{n}\sum_{i=1}^n Q^{(1)}_{j}(\boldsymbol{X}_i;\theta; t)-E_{\theta}\left\{Q^{(1)}_{j}(\boldsymbol{X}_1;\theta; t)\right\} \right\|_{_\mathcal{H}}>\varepsilon\right]\notag\\[.2 cm]
    &\hspace{43mm}\leq \frac{1}{n\, \varepsilon^2}\,E_{\theta}\left[\left\|Q^{(1)}_{j} (\boldsymbol{X}_1;\theta; t)\right\|_{_\mathcal{H}}^2\right]\to 0,\ \ j=1,2,\ldots,5.\notag
    \end{align}
    Then,
    \[
    \frac{1}{n}\sum_{i=1}^n Q^{(1)}(\boldsymbol{X}_i;\theta; t)\mathop{\longrightarrow} \limits^{P}E_{\theta}\left\{Q^{(1)}(\boldsymbol{X}_1;\theta; t)\right\},
    \]
    where $E_{\theta}\left\{Q^{(1)}(\boldsymbol{X}_1;\theta; t)\right\}=-v(t;\theta)\left(\lambda,\frac{1}{2}\lambda^2,\eta(t_1-1),\eta(t_2-1),\eta (t_1t_2-1)\right).$\\

    As $\|q_{n}\|_{\mathcal{H}}=o_{_P}(1)$, then using Assumption \ref{Sup1}, (\ref{Z-kn-aprox}) can be written as
    \[
    V_n(t;\hat{\theta}_n)=S_{n}(t;\theta)+s_{n},
    \]
    where $\|s_{n}\|_{\mathcal{H}}=o_{_P}(1)$, and
    \[S_n(t;\theta)=\frac{1}{\sqrt{n}}\sum_{i=1}^n \left[V(\boldsymbol{X}_i;\theta; t)+E_{\theta}\left\{Q^{(1)}(\boldsymbol{X}_1;\theta; t)\right\}  \boldsymbol{\ell}\left(\boldsymbol{X}_i; \theta\right)^{\top}\right].\]

    On the other hand, observe that
    \[
    \|S_{n}(\theta)\|_{_\mathcal{H}}^{2}=\frac{1}{n} \sum_{i=1}^n\sum_{j=1}^n h(\boldsymbol{X}_i,\boldsymbol{X}_j;\theta),\]
    where $h(\boldsymbol{x},\boldsymbol{y};\theta)$ is defined in (\ref{h}) and satisfies $\,h(\boldsymbol{x},\boldsymbol{y};\theta)= h(\boldsymbol{y},\boldsymbol{x};\theta)$,
    $E_{\theta}\left\{h^2(\boldsymbol{X}_1,\boldsymbol{X}_2;\theta)\right\}< \infty$, $\,E_{\theta} \left\{|h(\boldsymbol{X}_1,\boldsymbol{X}_1;\theta)|\right\} < \infty$ and $\,E_{\theta} \left\{h(\boldsymbol{X}_1,\boldsymbol{X}_2;\theta)\right\} = 0$. Thus, from Theorem 6.4.1.B in Serfling \citep{Ser80},
    \[
    \|S_{n}(\theta)\|_{_\mathcal{H}}^{2}
    \mathop{\longrightarrow} \limits^{\!L}\ \sum_{j\geq 1}\lambda_j\,\chi^2_{1j}\,,
    \]
    where $\chi^2_{11},\chi^2_{12},\ldots$ and the set $\{\lambda_j\}$ are as defined in the statement of the Theorem. In particular, $\|S_{n}(\theta)\|_{_\mathcal{H}}^{2}=O_P(1)$, which implies (\ref{lambdas}).
    \qed 
\end{proof}

The asymptotic null distribution of $V_{n, w}(\hat{\theta}_n)$ depends on the unknown true value of the parameter $\theta$, therefore, in practice, they do not provide a useful solution to the problem of estimating the null distribution of the respective statistical tests. This could be solved by replacing $\theta$ with $\hat{\theta}$.\\

But a greater difficulty is to determine the sets $\{\lambda_j\}_{j \geq 1}$, most of the cases, calculating the eigenvalues of an operator is not a simple task and in our case, we must also obtain the expression $h(\boldsymbol{x}, \boldsymbol{y}; \theta)$, which is not easy to find, since it depends on the function $\boldsymbol{\ell}$, which usually does not have a simple expression.\\

Thus, in the next section we consider another way to approximate the null distribution of the statistical test, the parametric bootstrap method.\\

\section{The bootstrap estimator}\label{bootstrap}
An alternative way to estimate the null distribution is through the parametric bootstrap method.\\

Let $\boldsymbol{X}_{1}, \ldots, \boldsymbol{X}_{n}$ be iid taking
values in $\mathbb{N}_0^2$. Assume that
$\hat{\theta}_n=\hat{\theta}_n(\boldsymbol{X}_1,
\ldots,\boldsymbol{X}_n)\in \Theta$. Let $\boldsymbol{X}^*_{1},
\ldots, \boldsymbol{X}^*_{n}$ be iid from a population with
distribution $BH(\hat{\theta}_{n})$, given $\boldsymbol{X}_{1},
\ldots, \boldsymbol{X}_{n}$, and let $V^*_{n,w}(\hat{\theta}_n^*)$
be the bootstrap version of $V_{n,w}(\hat{\theta}_n)$ obtained by
replacing $\boldsymbol{X}_{1}, $ $\ldots, \boldsymbol{X}_{n}$ and
$\hat{\theta}_n  =  \hat{\theta}_n (\boldsymbol{X}_{1}, \ldots,
\boldsymbol{X}_{n})$ by $\boldsymbol{X}^*_{1}, \ldots,
\boldsymbol{X}^*_{n}$ and $ \hat{\theta}_n^{*}=
\hat{\theta}_n(\boldsymbol{X}^*_{1}, \ldots, \boldsymbol{X}^*_{n})$,
respectively, in the expression of $V_{n,w}(\hat{\theta}_n)$. Let
$P_*$ denote the bootstrap conditional probability law, given
$\boldsymbol{X}_{1}, \ldots, \boldsymbol{X}_{n}$. In order to show
that the bootstrap  consistently estimate the null distribution of
$V_{n,w}(\hat{\theta}_n)$ we will assume the following assumption,
which is a bit stronger than Assumption 1.
\begin{assumption}\label{Sup2} Assumption 1 holds and the functions $\boldsymbol{\ell}$ and $J$ satisfy,
    \begin{enumerate}
        \item [(1)]  \vspace{-2pt} $\sup_{\vartheta\in\Theta_0} E_{\vartheta}\left[\|\boldsymbol{\ell} (\boldsymbol{X}; \vartheta)\|^2 I\left\{\|\boldsymbol{\ell}(\boldsymbol{X}; \vartheta)\|>\gamma \right\}\right]\longrightarrow 0$, as $\gamma \to \infty$, where $\Theta_0 \subseteq \Theta$ is an open neighborhood of $\theta$.
        \item [(2)] \vspace{3pt} $\boldsymbol{\ell} (\boldsymbol{X}; \vartheta)$ is continuous as function of $\vartheta$ at $\vartheta=\theta$ and $J(\vartheta)$ is finite $\forall \vartheta \in \Theta_0$.
    \end{enumerate}
\end{assumption}

As stated after Assumption 1,  Assumption  2 is not restrictive since it is fulfilled by commonly used estimators.\\

The next theorem shows that the bootstrap distribution of
$V_{n,w}(\hat{\theta}_n)$ consistently estimates its null
distribution.

\begin{theorem}\label{boot}
    Let $\boldsymbol{X}_{1},\ldots, \boldsymbol{X}_{n}$ be iid from a random vector $\boldsymbol{X}=(X_{1},X_{2}) \in \mathbb{N}_0^2$. Suppose that  Assumption 2 holds and that  $ \hat{\theta}_n=\theta+o(1)$, for some $\theta\in \Theta$. Then
    $$\sup_{x\in\mathbb{R}} \left|P_*\left\{V^*_{n,w}(\hat{\theta}^{*}_n) \leq x\right\}-P_{\theta}\left\{V_{n,w} (\hat{\theta}_n )\leq x\right\}\right|\ \mathop{\longrightarrow}\limits^{a.s.}\ 0.$$
\end{theorem}
\begin{proof}
    By definition, $V^*_{n,w}(\hat{\theta}_n^*)= \|V^*_n(\hat{\theta}_n^*)\|_{_\mathcal{H}}^{2}$, with
    \[V^*_n(t;\hat{\theta}_n^*)=\frac{1}{\sqrt{n}}\sum_{i=1}^n V(\boldsymbol{X}_i^*;\hat{\theta}_n^*;t)\]
    and $V(\boldsymbol{X};\theta; t)$ defined in (\ref{def_V}).\\

    Following similar steps to those given in the proof of Theorem \ref{ConvDebil} it can be seen that
    $V_{n,w}^*(\hat{\theta}_n^*)= \|W_{n}^*\|_{\mathcal{H}}^{2}+o_{P_*}(1)$,  where $W_{n}^*(t)$ is defined as $W_{n}(t)$ with
    $\boldsymbol{X}_i$ and $\theta$ replaced by $\boldsymbol{X}^*_i$ and $\hat{\theta}_n$, respectively.\\

    To derive the result, first we will check that assumptions (i)–(iii) in Theorem 1.1 of Kundu et al. \citep{Kundu} hold.\\

    Observe that
    \[Y_n^*(t)=\sum_{i=1}^n Y_{ni}^*(t)\]
    where
    \[Y_{ni}^*(t)= \frac{1}{\sqrt{n}}\,V^0(\boldsymbol{X}_i^*;\hat{\theta}_n; t),\ \ i=1,\ldots, n,\]

    Clearly $E_{*}\left\{Y_{ni}^*\right\}=0$ and $E_{*}\left\{\|Y_{ni}^*\|_{_\mathcal{H}}^{2}\right\}<\infty$. Let $K_n$ be the covariance kernel of $Y^{*}_n$ , which by SLLN satisfies
    \begin{align}
    K_n(u,v) &= E_*\!\left\{Y_n^*(u) Y_n^*(v)\right\}\notag\\
    &= E_*\left\{V^0(\boldsymbol{X}_1^*;\hat{\theta}_n; u) V^0(\boldsymbol{X}_1^*;\hat{\theta}_n; v)\right\}\notag\\
    &\hspace{4mm} \mathop{\longrightarrow}\limits^{a.s.} E_{\theta}\left\{V^0(\boldsymbol{X}_1;\theta; u) V^0(\boldsymbol{X}_1;\theta; v)\right\} = K(u,v).\notag
    \end{align}

    Moreover, let $\mathcal{Z}$ be zero-mean Gaussian process on $\mathcal{H}$ whose operator of covariance $C$ is characterized by
    \begin{align}
    \langle C f, h\rangle_{_{\mathcal{H}}}&=cov\left(\langle \mathcal{Z}, f\rangle_{_{\mathcal{H}}}, \langle \mathcal{Z}, h\rangle_{_{\mathcal{H}}}\right)\notag\\
    &=\int_{[0,1]^4} K(u,v)f(u)h(v) w(u)w(v)dudv.\notag
    \end{align}

    From the central limit theorem in Hilbert spaces (see, for example, van der Vaart and Wellner \citep{van}), it follows that $Y_n= \frac{1}{\sqrt{n}}\sum_{i=1}^n V^0(\boldsymbol{X}_i;\theta; t)\stackrel{L}{\longrightarrow} \mathcal{Z}$ on $\mathcal{H}$, when the data are iid from the random vector $\boldsymbol{X}\sim HB(\theta)$.\\

    Let $C_n$ denote the covariance operator of $Y_n^*$ and let $\{e_k: \, k \geq 0\}$ be an orthonormal basis of $\mathcal{H}$. Let $f,h \in \mathcal{H}$, by dominated convergence theorem,
    \begin{align} \lim_{n\to \infty}\,\langle C_n e_k, e_l\rangle_{_{\mathcal{H}}} &=\lim_{n\to \infty}\int_{[0,1]^4} K_n(u,v)e_k(u)e_l(v)w(u)w(v)du dv\notag\\ 
    &=\langle C e_k, e_l\rangle_{_{\mathcal{H}}}.\notag
    \end{align}

    Setting $a_{kl}= \langle C e_k, e_l\rangle_{_{\mathcal{H}}}$ in the aforementioned Theorem 1.1, this proves that condition (i) holds. To verify condition (ii), by
    using monotone convergence theorem, Parseval’s relation and dominated convergence theorem, we get
    \begin{align} \lim_{n\to \infty}\sum_{k=0}^{\infty} \langle C_{n} e_k,e_k\rangle_{_{\mathcal{H}}} &= \lim_{n\to \infty}\sum_{k=0}^{\infty}  \int_{[0,1]^4} K_n(u,v)e_k(u)e_k(v)w(u)w(v)du dv\notag\\ &=\sum_{k=0}^{\infty}  \int_{[0,1]^4} K(u,v)e_k(u)e_k(v)w(u)w(v)du dv=\sum_{k=0}^{\infty} \langle Ce_k,e_k\rangle_{_{\mathcal{H}}}\notag\\ &= \sum_{k=0}^{\infty} a_{kk}=\sum_{k=0}^{\infty}  E_{\theta}\!\left\{\langle \mathcal{Z},e_k\rangle^2_{_{\mathcal{H}_1}}\right\}= E_{\theta}\!\left\{\|\mathcal{Z}\|_{_{\mathcal{H}}}^{\,2}\right\}< \infty.\notag \end{align}

    To prove condition (iii), we first notice that
    \[\left|\langle Y_{ni}^*,e_k \rangle_{_{\mathcal{H}}}\right| \leq \frac{M}{\sqrt{n}},\ i=1,\ldots, n,\ \forall n, \ \text{where} \ 0<M<\infty.\]
    From the above inequality, for each fixed $\varepsilon > 0$,
    \[E_*\left[\langle Y_{ni}^*,e_k\rangle^2_{_{\mathcal{H}}}\ I\left\{\left|\langle Y_{ni}^*,e_k \rangle_{_{\mathcal{H}}}\right|>\varepsilon\right\}\right]=0.\]
    for sufficiently large $n$. This proves condition (iii). Therefore, $Y^*_n\stackrel{L}{\longrightarrow} \mathcal{Z}$ in $\mathcal{H}$, a.s. Now the result follows from the continuous mapping theorem.
    \qed 
\end{proof}

From Theorem \ref{boot}, the test function
$$\Psi^*_V=\left\{
\begin{array}{ll}
1, & \text{if}\ V^*_{n,w}(\hat{\theta}^*_n)\geq v^*_{n,w,\alpha}, \\
0, & \text{otherwise},
\end{array}
\right.$$ or equivalently, the test that rejects $H_0$ when
$p^*=P_*\{V^*_{n,w}(\hat{\theta}^{*}_n) \geq V_{obs}\}\leq \alpha,$
is asymptotically correct in the sense that when $H_0$ is true,
$\lim P_{\theta}(\Psi^*_V=1)=\alpha$, where
$v^*_{n,w,\alpha}=\inf\{x:P_*(V^*_{n,w} (\hat{\theta}^*_n)\geq
x)\leq \alpha\}$ is the $\alpha$ upper percentile of the bootstrap
distribution of $V_{n,w} (\hat{\theta}_n)$ and $V_{obs}$ is the
observed value of the test statistic.

\section{Numerical results}\label{NumericalResults}

According to Novoa-Muñoz and Jiménez-Gamero in \citep{Nov}, the properties of the statistic $V_{n,w}(\hat{{\theta}}_n)$ are asymptotic, that is, such properties describe the behavior of the test proposed for large samples. To study the goodness of the bootstrap approach for samples of finite size, a simulation experiment was carried out. In this section we describe this experiment and provide a summary of the results that have been obtained.\\

All computer calculations made in this paper were carried out
through the use of programs written in the R language \cite{R}.

To calculate $V_{n,w}(\hat{\theta}_n)$ it is necessary to give an
explicit form to the weight function $w$. Here the following is
taken into account
\begin{equation}\label{funcion-peso-explicita}
w(t;a_1,a_2)=t_1^{a_1}t_2^{a_2}.
\end{equation}

Observe that the only restrictions that have been imposed on the
weight function are that $w$ be positive almost everywhere in
$[0,1]^2$ and the established in (\ref{int-funcion-peso}). The
function $w(t;a_1,a_2)$ given in (\ref{funcion-peso-explicita})
meets these conditions whenever $a_i>-1$, $i=1,2$. Hence
\[
V_{n,w}(\hat{\theta}_n)=n\int_{0}^{1}\int_{0}^{1}\left[\sum_{i=1}^{n}t_1^{X_{i1}}t_2^{X_{i2}}-exp\left(\hat{\mu}\hat{\lambda}+\frac{1}{2}\hat{\sigma}^2\hat{\lambda}^2\right)\right]^2t_1^{a_1}\,t_2^{a_2}\,dt_1dt_2.
\]
It was not possible to find an explicit form of the statistic
$V_{n,w}(\hat{\theta}_n)$, for which, its calculation was used the
curvature package of R \cite{R} to calculate it.

\subsection{Simulated data}

In order to approximate  of the null distribution of the statistic
$V_{n,w}(\hat{\theta}_n)$ for finite-size samples samples of size
$30$ , $50$ and $70$ from a $BH(\theta)$, for
$\theta=(\mu,\sigma^2,\lambda_1,\lambda_2,\lambda_3)$,
using the pgf (\ref{fgp}), with $\lambda_3=0$ were utilised. The combinations of parameters were chosen in such a way that $\mu>\sigma^2(\lambda_i+\lambda_3)$, $i=1,2$. \\

The selected values of the other parameters were $\mu\in\{1.0,1.5,2.0\}$, $\sigma^2\in\{0.8,1.0\}$, $\lambda_1\in\{0.10,0.25,0.50,0.75,1.00\}$ and $\lambda_2\in\{0.20,0.25,0.50,0.75\}$.\\

The selected values of $\lambda_1$ and $\lambda_2$ were not greater than 1 since the Hermite distribution is characterized as being zero-inflated.\\

To estimate the parameter $\theta$ we use the maximum likelihood
method given in Kocherlakota and Kocherlakota \citep{Kocher}. Then
we approximated the bootstrap $p-$values of the proposed test with
weight function given in (\ref{funcion-peso-explicita}) for
$(a_1,a_2)\in \{(0,0), (1,0),(0,1),(1,1),(5,1),(1,5),(5,5)\}$ and we
generate $B=500$
bootstrap samples.\\

The above procedure was repeated 1000 times and the fraction of the estimated $p-$values that were found to be less than or equal to 0.05 and 0.10, which are the estimates type I error probabilities for $\alpha=$ 0.05 and 0.1.\\

The results obtained are presented in Tables
\ref{my-label1}-\ref{my-label7} for the different pairs $(a_1,
a_2)$. In each table, the established order was growing in $\mu$ and
$\sigma^2$, and for each new $\mu$ increasing values in $\lambda_1$,
and in each new $\lambda_1$, increasing values for $\lambda_2$.
From these results we can conclude that the parametric bootstrap method provides good approximations to the null distribution of the $V_{n,w}(\hat{\theta}_n)$ in most of the cases considered.\\

It is seen that the values of $a_1$ and $a_2$ of the weight function
affects bootstrap estimates of $p-$values.

From the tables it is clear that the bootstrap $p-$values are increasingly approaching the nominal value as $n$ increases. These approximations are better when $a_1=a_2$. In particular, when $a_1=a_2$ are small (less than 5), then the bootstrap p-values are approached from the left (below) to the nominal value, otherwise it happens when $a_1=a_2$ are fairly large values (greater or equal to 5). Table \ref{my-label4} is the one that shows the best results, being the weight function with $a_1=a_2=1$ which presents the best $p-$values estimates.\\

Unfortunately we could not find a closed form for our statistic
$V_{n,w}(\hat{\theta}_n)$, so to calculate it we used the curvature
package of the software R \cite{R}. This had a serious impact on the
computation time since the simulations were increased in their
execution time by at least 30\%.
\begin{table}
    \caption{Simulation results for the probability of type I error for $a_1=0$ and $a_2=0$.}
    \label{my-label1}
    \begin{tabular}{ccccccc}
        \hline\noalign{\smallskip}
        & \multicolumn{2}{c}{\textbf{$n=30$}}           & \multicolumn{2}{c}{\textbf{$n=50$}}           & \multicolumn{2}{c}{\textbf{$n=70$}}           \\
        \cline{2-7} \noalign{\smallskip}
        $\theta$ &  $\alpha$=0.05 & $\alpha$=0.1 & $\alpha$=0.05 & $\alpha$=0.1 & $\alpha$=0.05 & $\alpha$=0.1 \\
        \noalign{\smallskip}\hline\noalign{\smallskip}
        (1.0,0.8,0.10,0.20,0.00) & 0.012 & 0.053 & 0.029 & 0.069 & 0.037 & 0.081 \\
        (1.0,0.8,0.25,0.25,0.00) & 0.027 & 0.067 & 0.037 & 0.064 & 0.043 & 0.094 \\
        (1.0,0.8,0.50,0.20,0.00) & 0.016 & 0.062 & 0.046 & 0.073 & 0.047 & 0.087 \\
        (1.0,0.8,0.50,0.50,0.00) & 0.025 & 0.063 & 0.042 & 0.076 & 0.044 & 0.091 \\ (1.5,1.0,0.50,0.50,0.00) & 0.010 & 0.064 & 0.035 & 0.078 & 0.042 & 0.089 \\
        (1.5,1.0,0.50,0.75,0.00) & 0.010 & 0.065 & 0.036 & 0.084 & 0.041 & 0.084 \\
        (1.5,1.0,0.75,0.25,0.00) & 0.017 & 0.071 & 0.038 & 0.087 & 0.043 & 0.088 \\
        (1.5,1.0,1.00,0.25,0.00) & 0.027 & 0.076 & 0.039 & 0.090 & 0.042 & 0.092 \\
        (2.0,1.0,0.25,0.75,0.00) & 0.017 & 0.067 & 0.038 & 0.082 & 0.047 & 0.089 \\
        (2.0,1.0,0.50,0.25,0.00) & 0.011 & 0.067 & 0.037 & 0.088 & 0.045 & 0.091 \\
        (2.0,1.0,0.75,0.25,0.00) & 0.029 & 0.070 & 0.035 & 0.087 & 0.043 & 0.089 \\
        \noalign{\smallskip}
        \hline
    \end{tabular}
\end{table}
\begin{table}
    \caption{Simulation results for the probability of type I error for $a_1=1$ and $a_2=0$.}
    \label{my-label2}
    \begin{tabular}{ccccccc}
        \hline\noalign{\smallskip}
        & \multicolumn{2}{c}{\textbf{$n=30$}}           & \multicolumn{2}{c}{\textbf{$n=50$}}           & \multicolumn{2}{c}{\textbf{$n=70$}}           \\
        \cline{2-7} \noalign{\smallskip}
        $\theta$ &  $\alpha$=0.05 & $\alpha$=0.1 & $\alpha$=0.05 & $\alpha$=0.1 & $\alpha$=0.05 & $\alpha$=0.1 \\
        \noalign{\smallskip}\hline\noalign{\smallskip}
        (1.0,0.8,0.10,0.20,0.00) & 0.010 & 0.039 & 0.025 & 0.073 & 0.043 & 0.088 \\
        (1.0,0.8,0.25,0.25,0.00) & 0.025 & 0.073 & 0.037 & 0.088 & 0.041 & 0.104 \\
        (1.0,0.8,0.50,0.20,0.00) & 0.027 & 0.072 & 0.041 & 0.083 & 0.045 & 0.086 \\
        (1.0,0.8,0.50,0.50,0.00) & 0.035 & 0.053 & 0.042 & 0.072 & 0.045 & 0.101 \\ (1.5,1.0,0.50,0.50,0.00) & 0.011 & 0.064 & 0.031 & 0.080 & 0.038 & 0.085 \\
        (1.5,1.0,0.50,0.75,0.00) & 0.019 & 0.065 & 0.034 & 0.078 & 0.039 & 0.080 \\
        (1.5,1.0,0.75,0.25,0.00) & 0.025 & 0.081 & 0.038 & 0.085 & 0.042 & 0.084 \\
        (1.5,1.0,1.00,0.25,0.00) & 0.037 & 0.074 & 0.035 & 0.085 & 0.040 & 0.086 \\
        (2.0,1.0,0.25,0.75,0.00) & 0.027 & 0.071 & 0.034 & 0.082 & 0.047 & 0.089 \\
        (2.0,1.0,0.50,0.25,0.00) & 0.011 & 0.077 & 0.031 & 0.084 & 0.044 & 0.086 \\
        (2.0,1.0,0.75,0.25,0.00) & 0.019 & 0.080 & 0.035 & 0.085 & 0.044 & 0.087 \\ \noalign{\smallskip}
        \hline
    \end{tabular}
\end{table}
\begin{table}
    \caption{Simulation results for the probability of type I error for $a_1=0$ and $a_2=1$.}
    \label{my-label3}
    \begin{tabular}{ccccccc}
        \hline\noalign{\smallskip}
        & \multicolumn{2}{c}{\textbf{$n=30$}}           & \multicolumn{2}{c}{\textbf{$n=50$}}           & \multicolumn{2}{c}{\textbf{$n=70$}}           \\
        \cline{2-7} \noalign{\smallskip}
        $\theta$ &  $\alpha$=0.05 & $\alpha$=0.1 & $\alpha$=0.05 & $\alpha$=0.1 & $\alpha$=0.05 & $\alpha$=0.1 \\
        \noalign{\smallskip}\hline\noalign{\smallskip}
        (1.0,0.8,0.10,0.20,0.00) & 0.014 & 0.044 & 0.029 & 0.067 & 0.043 & 0.088 \\
        (1.0,0.8,0.25,0.25,0.00) & 0.028 & 0.068 & 0.039 & 0.079 & 0.042 & 0.084 \\
        (1.0,0.8,0.50,0.20,0.00) & 0.019 & 0.063 & 0.042 & 0.083 & 0.057 & 0.092 \\
        (1.0,0.8,0.50,0.50,0.00) & 0.029 & 0.063 & 0.045 & 0.075 & 0.054 & 0.089 \\ (1.5,1.0,0.50,0.50,0.00) & 0.011 & 0.066 & 0.039 & 0.079 & 0.042 & 0.089 \\
        (1.5,1.0,0.50,0.75,0.00) & 0.013 & 0.070 & 0.043 & 0.082 & 0.043 & 0.087 \\
        (1.5,1.0,0.75,0.25,0.00) & 0.017 & 0.081 & 0.042 & 0.089 & 0.043 & 0.092 \\
        (1.5,1.0,1.00,0.25,0.00) & 0.037 & 0.086 & 0.045 & 0.091 & 0.045 & 0.093 \\
        (2.0,1.0,0.25,0.75,0.00) & 0.047 & 0.077 & 0.048 & 0.084 & 0.047 & 0.089 \\
        (2.0,1.0,0.50,0.25,0.00) & 0.014 & 0.077 & 0.037 & 0.089 & 0.043 & 0.093 \\
        (2.0,1.0,0.75,0.25,0.00) & 0.027 & 0.080 & 0.041 & 0.097 & 0.044 & 0.096 \\  \noalign{\smallskip}
        \hline
    \end{tabular}
\end{table}
\begin{table}
    \caption{Simulation results for the probability of type I error for $a_1=1$ and $a_2=1$.}
    \label{my-label4}
    \begin{tabular}{ccccccc}
        \hline\noalign{\smallskip}
        & \multicolumn{2}{c}{\textbf{$n=30$}}           & \multicolumn{2}{c}{\textbf{$n=50$}}           & \multicolumn{2}{c}{\textbf{$n=70$}}           \\
        \cline{2-7} \noalign{\smallskip}
        $\theta$ &  $\alpha$=0.05 & $\alpha$=0.1 & $\alpha$=0.05 & $\alpha$=0.1 & $\alpha$=0.05 & $\alpha$=0.1 \\
        \noalign{\smallskip}\hline\noalign{\smallskip}
        (1.0,0.8,0.10,0.20,0.00) & 0.016 & 0.073 & 0.024 & 0.086 & 0.048 & 0.092 \\
        (1.0,0.8,0.25,0.25,0.00) & 0.032 & 0.058 & 0.037 & 0.088 & 0.049 & 0.091 \\
        (1.0,0.8,0.50,0.20,0.00) & 0.024 & 0.064 & 0.043 & 0.085 & 0.048 & 0.089 \\
        (1.0,0.8,0.50,0.50,0.00) & 0.033 & 0.072 & 0.043 & 0.086 & 0.049 & 0.093 \\ (1.5,1.0,0.50,0.50,0.00) & 0.030 & 0.072 & 0.038 & 0.088 & 0.046 & 0.090 \\
        (1.5,1.0,0.50,0.75,0.00) & 0.033 & 0.071 & 0.042 & 0.084 & 0.047 & 0.098 \\
        (1.5,1.0,0.75,0.25,0.00) & 0.036 & 0.097 & 0.039 & 0.097 & 0.049 & 0.099 \\
        (1.5,1.0,1.00,0.25,0.00) & 0.039 & 0.088 & 0.046 & 0.090 & 0.049 & 0.093 \\
        (2.0,1.0,0.25,0.75,0.00) & 0.031 & 0.087 & 0.044 & 0.092 & 0.048 & 0.099 \\
        (2.0,1.0,0.50,0.25,0.00) & 0.035 & 0.068 & 0.039 & 0.081 & 0.047 & 0.093 \\
        (2.0,1.0,0.75,0.25,0.00) & 0.037 & 0.080 & 0.045 & 0.088 & 0.049 & 0.096 \\  \noalign{\smallskip}
        \hline
    \end{tabular}
\end{table}
\begin{table}
    \caption{Simulation results for the probability of type I error for $a_1=1$ and $a_2=5$.}
    \label{my-label5}
    \begin{tabular}{ccccccc}
        \hline\noalign{\smallskip}
        & \multicolumn{2}{c}{\textbf{$n=30$}}           & \multicolumn{2}{c}{\textbf{$n=50$}}           & \multicolumn{2}{c}{\textbf{$n=70$}}           \\
        \cline{2-7} \noalign{\smallskip}
        $\theta$ &  $\alpha$=0.05 & $\alpha$=0.1 & $\alpha$=0.05 & $\alpha$=0.1 & $\alpha$=0.05 & $\alpha$=0.1 \\
        \noalign{\smallskip}\hline\noalign{\smallskip}
        (1.0,0.8,0.10,0.20,0.00) & 0.014 & 0.037 & 0.032 & 0.075 & 0.051 & 0.093 \\
        (1.0,0.8,0.25,0.25,0.00) & 0.023 & 0.074 & 0.053 & 0.090 & 0.060 & 0.113 \\
        (1.0,0.8,0.50,0.20,0.00) & 0.036 & 0.101 & 0.062 & 0.110 & 0.064 & 0.117 \\
        (1.0,0.8,0.50,0.50,0.00) & 0.023 & 0.080 & 0.042 & 0.107 & 0.063 & 0.109 \\ (1.5,1.0,0.50,0.50,0.00) & 0.022 & 0.081 & 0.037 & 0.111 & 0.046 & 0.108 \\
        (1.5,1.0,0.50,0.75,0.00) & 0.039 & 0.095 & 0.048 & 0.108 & 0.056 & 0.108 \\
        (1.5,1.0,0.75,0.25,0.00) & 0.034 & 0.108 & 0.048 & 0.107 & 0.054 & 0.108 \\
        (1.5,1.0,1.00,0.25,0.00) & 0.037 & 0.107 & 0.059 & 0.109 & 0.054 & 0.107 \\
        (2.0,1.0,0.25,0.75,0.00) & 0.048 & 0.106 & 0.056 & 0.108 & 0.054 & 0.106 \\
        (2.0,1.0,0.50,0.25,0.00) & 0.025 & 0.107 & 0.047 & 0.108 & 0.045 & 0.108 \\
        (2.0,1.0,0.75,0.25,0.00) & 0.043 & 0.107 & 0.045 & 0.107 & 0.043 & 0.106 \\   \noalign{\smallskip}
        \hline
    \end{tabular}
\end{table}
\begin{table}
    \caption{Simulation results for the probability of type I error for $a_1=5$ and $a_2=1$.}
    \label{my-label6}
    \begin{tabular}{ccccccc}
        \hline\noalign{\smallskip}
        & \multicolumn{2}{c}{\textbf{$n=30$}}           & \multicolumn{2}{c}{\textbf{$n=50$}}           & \multicolumn{2}{c}{\textbf{$n=70$}}           \\
        \cline{2-7} \noalign{\smallskip}
        $\theta$ &  $\alpha$=0.05 & $\alpha$=0.1 & $\alpha$=0.05 & $\alpha$=0.1 & $\alpha$=0.05 & $\alpha$=0.1 \\
        \noalign{\smallskip}\hline\noalign{\smallskip}
        (1.0,0.8,0.10,0.20,0.00) & 0.015 & 0.040 & 0.032 & 0.062 & 0.042 & 0.081 \\
        (1.0,0.8,0.25,0.25,0.00) & 0.034 & 0.076 & 0.045 & 0.101 & 0.048 & 0.104 \\
        (1.0,0.8,0.50,0.20,0.00) & 0.028 & 0.084 & 0.048 & 0.073 & 0.053 & 0.089 \\
        (1.0,0.8,0.50,0.50,0.00) & 0.028 & 0.069 & 0.045 & 0.079 & 0.054 & 0.098 \\ (1.5,1.0,0.50,0.50,0.00) & 0.019 & 0.071 & 0.035 & 0.078 & 0.042 & 0.099 \\
        (1.5,1.0,0.50,0.75,0.00) & 0.044 & 0.104 & 0.048 & 0.098 & 0.056 & 0.104 \\
        (1.5,1.0,0.75,0.25,0.00) & 0.027 & 0.107 & 0.038 & 0.105 & 0.046 & 0.103 \\
        (1.5,1.0,1.00,0.25,0.00) & 0.037 & 0.117 & 0.043 & 0.112 & 0.060 & 0.107 \\
        (2.0,1.0,0.25,0.75,0.00) & 0.037 & 0.112 & 0.039 & 0.108 & 0.054 & 0.108 \\
        (2.0,1.0,0.50,0.25,0.00) & 0.026 & 0.077 & 0.034 & 0.109 & 0.055 & 0.109 \\
        (2.0,1.0,0.75,0.25,0.00) & 0.034 & 0.116 & 0.045 & 0.107 & 0.056 & 0.105 \\    \noalign{\smallskip}
        \hline
    \end{tabular}
\end{table}
\begin{table}
    \caption{Simulation results for the probability of type I error for $a_1=5$ and $a_2=5$.}
    \label{my-label7}
    \begin{tabular}{ccccccc}
        \hline\noalign{\smallskip}
        & \multicolumn{2}{c}{\textbf{$n=30$}}           & \multicolumn{2}{c}{\textbf{$n=50$}}           & \multicolumn{2}{c}{\textbf{$n=70$}}           \\
        \cline{2-7} \noalign{\smallskip}
        $\theta$ &  $\alpha$=0.05 & $\alpha$=0.1 & $\alpha$=0.05 & $\alpha$=0.1 & $\alpha$=0.05 & $\alpha$=0.1 \\
        \noalign{\smallskip}\hline\noalign{\smallskip}
        (1.0,0.8,0.10,0.20,0.00) & 0.017 & 0.035 & 0.032 & 0.065 & 0.050 & 0.089 \\
        (1.0,0.8,0.25,0.25,0.00) & 0.027 & 0.077 & 0.034 & 0.081 & 0.043 & 0.084 \\
        (1.0,0.8,0.50,0.20,0.00) & 0.030 & 0.086 & 0.042 & 0.087 & 0.048 & 0.104 \\
        (1.0,0.8,0.50,0.50,0.00) & 0.013 & 0.069 & 0.030 & 0.076 & 0.045 & 0.105 \\ (1.5,1.0,0.50,0.50,0.00) & 0.016 & 0.063 & 0.035 & 0.078 & 0.046 & 0.087 \\
        (1.5,1.0,0.50,0.75,0.00) & 0.019 & 0.085 & 0.061 & 0.089 & 0.054 & 0.094 \\
        (1.5,1.0,0.75,0.25,0.00) & 0.031 & 0.071 & 0.053 & 0.102 & 0.047 & 0.098 \\
        (1.5,1.0,1.00,0.25,0.00) & 0.037 & 0.086 & 0.049 & 0.104 & 0.052 & 0.102 \\
        (2.0,1.0,0.25,0.75,0.00) & 0.015 & 0.087 & 0.057 & 0.098 & 0.055 & 0.101 \\
        (2.0,1.0,0.75,0.25,0.00) & 0.040 & 0.097 & 0.054 & 0.102 & 0.053 & 0.102 \\     \noalign{\smallskip}
        \hline
    \end{tabular}
\end{table}

\subsection{The power of a hypothesis test}

To study the power we repeated the previous experiment for samples
of size $n=50 $ and for the weight function we used the values of
$a_1$ and $a_2$ that yielded the best results in the study of type I
error. The alternative distributions we use are detailed below:
\begin{itemize}
    \item [\textbullet] bivariate binomial distribution $BB(m; p_1, p_2, p_3)$, where $p_1+p_2-p_3\leq 1$, $p_1\geq p_3$, $p_2\geq p_3$ and $p_3>0$,
    \item [\textbullet] \vskip .2 cm bivariate Poisson distribution $BP(\lambda_1,\lambda_2,\lambda_3)$, where $\lambda_1>\lambda_3$, $\lambda_2>\lambda_3>0$,
    \item [\textbullet] \vskip .2 cm bivariate logarithmic series distribution $BLS(\lambda_1,\lambda_2,\lambda_3)$, where $0<\lambda_1+\lambda_2+\lambda_3<1$,
    \item [\textbullet] \vskip .2 cm bivariate negative binomial distribution $BNB(\nu; \gamma_0, \gamma_1, \gamma_2)$, where $\nu\in \mathbb{N}, \gamma_0>\gamma_2,\gamma_1>\gamma_2$ and $\gamma_2>0$,
    \item [\textbullet] \vskip .2 cm bivariate Neyman type A distribution $BNTA(\lambda;\lambda_1,\lambda_2,\lambda_3)$, where $0<\lambda_1+\lambda_2+\lambda_3\leq 1$,
    \item [\textbullet] \vskip .2 cm bivariate Poisson distribution mixtures of the form $pBP(\theta)+(1-p)BP(\lambda)$, where $0<p<1$, denoted by $BPP(p;\theta,\lambda)$.\\
\end{itemize}
Table \ref{power1} displays the alternatives considered and the
estimated power for nominal significance level $\alpha=0.05$.
Analyzing this table we can conclude that all the considered tests,
denoted by $V_{(a_1,a_2)}$, are able to detect the alternatives
studied and with a good power, giving better results in cases where
$a_1=a_2$. The best result was achieved for $a_1=a_2=1$, as
expected, as occurred in the study of type I error.
\begin{table}
    \caption{Simulation results for the power. The values are in the form of percentages, rounded to    the nearest integer.}
    \label{power1}
    \begin{tabular}{lrrrrr}
        \hline\noalign{\smallskip}
        Alternative &  $V_{(0,0)}$ & $V_{(1,0)}$ & $V_{(1,1)}$ & $V_{(1,5)}$ & $V_{(5,5)}$ \\
        \hline
        $BB(1;0.41,0.02,0.01)$ &   87 & 81 & 89 & 81 & 85 \\
        $BB(1;0.41,0.03,0.02)$ &   85 & 82 & 88 & 80 & 86 \\
        $BB(2;0.61,0.01,0.01)$ &   93 & 84 & 98 & 83 & 92 \\
        $BB(1;0.61,0.03,0.02)$ &   95 & 89 & 100 & 87 & 95 \\
        $BB(2;0.71,0.01,0.01)$ &   94 & 86 & 100 & 85 & 93 \\ [.08 cm]
        $BP(1.00,1.00,0.25)$   &   85 & 76 & 89 & 77 & 82 \\
        $BP(1.00,1.00,0.50)$   &   84 & 77 & 91 & 72 & 85 \\
        $BP(1.00,1.00,0.75)$   &   87 & 75 & 92 & 73 & 83 \\
        $BP(1.50,1.00,0.31)$   &   87 & 77 & 93 & 75 & 87 \\
        $BP(1.50,1.00,0.92)$   &   86 & 76 & 92 & 77 & 87 \\ [.08 cm]
        $BLS(0.25,0.15,0.10)$   &  94 & 85 & 98 & 86 & 95 \\
        $BLS(5d/7,d/7,d/7)^*$   &  91 & 85 & 100 & 84 & 90 \\
        $BLS(3d/4,d/8,d/8)^*$   &  90 & 86 & 100 & 84 & 90 \\
        $BLS(7d/9,d/9,d/9)^*$   &  94 & 86 & 100 & 83 & 93 \\
        $BLS(0.51,0.01,0.02)$   &  90 & 83 & 98 & 83 & 91 \\ [.08 cm]
        $BNB(1;0.92,0.97,0.01)$  &  93 & 87 & 96 & 85 & 92\\
        $BNB(1;0.97,0.97,0.01)$  &  92 & 86 & 95 & 85 & 92\\
        $BNB(1;0.97,0.97,0.02)$  &  94 & 88 & 100 & 89 & 93\\
        $BNB(1;0.98,0.98,0.01)$  &  92 & 84 & 97 & 85 & 92\\
        $BNB(1;0.99,0.99,0.01)$  &  91 & 84 & 96 & 83 & 91\\ [.08 cm]
        $BNTA(0.21;0.01,0.01,0.98)$ & 93 & 86 & 98 & 85 & 92\\
        $BNTA(0.24;0.01,0.01,0.98)$ & 95 & 87 & 100 & 85 & 95\\
        $BNTA(0.26;0.01,0.01,0.97)$ & 93 & 85 & 97 & 86 & 93\\
        $BNTA(0.26;0.01,0.01,0.98)$ & 94 & 85 & 98 & 86 & 94\\
        $BNTA(0.28;0.01,0.01,0.97)$ & 93 & 86 & 96 & 86 & 94\\ [.08 cm]
        $BPP(0.31;(0.2,0.2,0.1),(1.0,1.0,0.9))$ & 76 & 70 & 82 & 72 & 77\\
        $BPP(0.31;(0.2,0.2,0.1),(1.0,1.2,0.9))$ & 77 & 71 & 84 & 71 & 76\\
        $BPP(0.32;(0.2,0.2,0.1),(1.0,1.0,0.9))$ & 78 & 71 & 84 & 71 & 76\\
        $BPP(0.33;(0.2,0.2,0.1),(1.0,1.0,0.9))$ & 78 & 70 & 85 & 70 & 77\\
        $BPP(0.33;(0.2,0.2,0.1),(1.0,1.1,0.9))$ & 76 & 71 & 83 & 70 & 78\\      \hline\noalign{\smallskip}
        \multicolumn{3}{l}{\normalsize $^*\, d=1-\exp(-1)\approx$ 0.63212.}
    \end{tabular}
\end{table}

\subsection{Real data set}

Now, the proposed test will be applied to a real data set. The data set comprises the number of accidents in two different years presented in \cite{Kocher}. Where $X$ is the accident number of the first period and $Y$ the accident number of the second period. Table \ref{datosreales} shows the real data set.\\
\begin{table}
    \caption{Real data of $X$ accident number in a period and $Y$ of another period.}
    \label{datosreales}
    \begin{tabular}{rrrrrrrrrrrrr}
        \hline\noalign{\smallskip}
        & & & \multicolumn{8}{c}{$X$} &\\
        & & & \multicolumn{1}{c}{0} & \multicolumn{1}{c}{1} & \multicolumn{1}{c}{2} & \multicolumn{1}{c}{3} & \multicolumn{1}{c}{4} & \multicolumn{1}{c}{5} & \multicolumn{1}{c}{6} & \multicolumn{1}{c}{7} & & Total \\
        \hline
        &0& & 117 & 96 & 55 & 19 & 2 & 2 & 0 & 0 & &291\\
        &1& & 61 & 69 & 47 & 27 & 8 & 5 & 1 & 0 & &218\\
        &2& & 34 & 42 & 31 & 13 & 7 & 2 & 3 & 0 & &132\\
        $Y$ &3& & 7 & 15 & 17 & 7 & 3 & 1 & 0 & 0 & &49\\
        &4& & 3 & 3 & 1 & 1 & 2 & 1 & 1 & 1 & &13\\
        &5& & 2 & 1 & 0 & 0 & 0 & 0 & 0 & 0 & &3\\
        &6& & 0 & 0 & 0 & 0 & 1 & 0 & 0 & 0 & &1\\
        &7& & 0 & 0 & 0 & 1 & 0 & 0 & 0 & 0 & &1\\
        \hline\noalign{\smallskip}
        &Total&& 224 & 226 & 150 & 68 & 23 & 11 & 5 & 1 & &708\\
    \end{tabular}
\end{table}

The $p-$value obtained from the statistic $V_{n,w}(\hat{\theta}_n)$
of the proposed test, with $a_1=1$ and $a_2=0$ applied to the real
values is 0.838, therefore, we decided not to reject the null
hypothesis, that is, the data seem to have a BHD. This is consistent
with the results presented by Kemp and Papageorgiou in \citep{kepa},
who performed the goodness-of-fit test $\chi^2$ obtaining a
$p-$value of 0.3078.

\section*{Acknowledgements}
    The authors would like to thank research projects DIUBB 172409 GI/C, DIUBB 192408 2/R and Fondo de Apoyo a la Participaci\'on a Eventos Internacionales (FAPEI) at Universidad del B\'io-B\'io, Chile.\\

\section*{Conflict of interest}

The authors declare that they have no conflict of interest.







\end{document}